\documentclass[11pt]{amsart}
\usepackage{etex} 

\usepackage{amsmath}
\usepackage{amssymb}
\usepackage{amsthm}
\usepackage{amsfonts}
\usepackage{booktabs}
\usepackage[usenames,dvipsnames,svgnames]{xcolor}
\usepackage{epsfig}
\usepackage{hyperref}
\usepackage{mathrsfs}
\usepackage{mathtools}
\usepackage{pictexwd, dcpic}
\usepackage{stmaryrd}
\usepackage{mathabx}
\usepackage{graphicx}
\usepackage{url}

\usepackage{tikz}
\usepackage{tikz-cd}
\usepackage{caption}
\usetikzlibrary{decorations.markings}

\newcommand{\CC}{\mathbb{C}}
\newcommand{\PP}{\mathbb{P}}
\newcommand{\ZZ}{\mathbb{Z}}

\newcommand{\Aut}{\operatorname{Aut}}
\newcommand{\Mod}{\operatorname{Mod}}
\newcommand{\SMod}{\operatorname{SMod}}

\newcommand{\sm}{\setminus}
\newcommand{\wt}[1]{\widetilde{#1}}

\newcommand{\abs}[1]{\left\lvert #1 \right\rvert}

\theoremstyle{definition}

\newtheorem*{ques}{Question}

\theoremstyle{plain}
\newtheorem{thm}{Theorem}[section]
\newtheorem{lem}[thm]{Lemma}
\newtheorem{cor}[thm]{Corollary}

\newcommand{\p}[1]{\smallskip\noindent {\bf #1:}} 

\title{Lifting Homeomorphisms and cyclic branched covers of spheres}
\author{Tyrone Ghaswala}
\address{Department of Pure Mathematics, University of Waterloo, Waterloo, ON, N2L 3G1,
Canada}
\email{ty.ghaswala@gmail.com}
\author{Rebecca R. Winarski}
\address{Department of Mathematical sciences, University of Wisconsin-Milwaukee, Milwaukee, WI 53211-3029, USA}
\email{rebecca.winarski@gmail.com}

\begin{document}
\begin{abstract}
We characterize the cyclic branched covers of the 2-sphere where every homeomorphism of the sphere lifts to a homeomorphism of the covering surface.  This answers a question that appeared in an early version of the erratum of Birman and Hilden \cite{erratum}.
\end{abstract}
\bibliographystyle{plain}
\maketitle
\section{Introduction}
Let $\Sigma$ be a closed orientable surface.  Let $p:\Sigma\rightarrow \Sigma_0$ be a cyclic branched covering space of the sphere $\Sigma_0$.  We will assume all homeomorphisms of $\Sigma_0$ preserve the set of branch points.  For brevity, we will say a homeomorphism $f$ of $\Sigma_0$ {\it lifts} if there exists a homeomorphism $\widetilde{f}$ of $\Sigma$ such that $p\widetilde{f}=fp$.

The 2-fold cover $\Sigma\rightarrow\Sigma_0$ induced by the hyperelliptic involution of $\Sigma$ is of interest in the study of low dimensional topology eg. Brendle--Margalit--Putman \cite{BMP}, Johnson--Schmoll \cite{JS}, Morifuji \cite{morifuji} and algebraic geometry eg. Gorchinskiy--Viviani \cite{GV}, Hidalgo \cite{hidalgo}.  In this 2-fold covering space, every homeomorphism of $\Sigma_0$ lifts.  However, it is not true in general that every homeomorphism of $\Sigma_0$ lifts to a homeomorphism of the covering surface.

We answer the question that appeared in an early version of \cite{erratum}:  
\begin{ques}[Birman--Hilden]\label{main_ques}
Let $\Sigma$ be a closed orientable surface, and $\Sigma_0$ a 2-sphere.  For which cyclic branched covering spaces of the sphere $\Sigma\rightarrow \Sigma_0$ does every homeomorphism of $\Sigma_0$ lift?
\end{ques}

Let $\Sigma_{0}^{\circ}$ denote a sphere with punctures, and $x_0\in\Sigma_0^{\circ}$.  Each characteristic subgroup of $\pi_1(\Sigma_0^{\circ},x_0)$ corresponds to a branched covering space of $\Sigma_0$ where every homeomorphism lifts.  
However, cyclic subgroups of $\pi_1(\Sigma_0^{\circ},x_0)$ are rarely characteristic in $\pi_1(\Sigma_0^{\circ},x_0)$.  In fact, in the case of the hyperelliptic involution, the corresponding subgroup of $\pi_1(\Sigma_0^{\circ},x_0)$ is not characteristic, but all homeomorphisms of $\Sigma_0$ lift.

Let $A$ be a finite abelian group.  An {\it admissible $k$-tuple} is a tuple $(a_1,\ldots,a_k) \in (A\sm\{0\})^k$ such that $\sum_{i=1}^k a_i = 0$ and $\{a_1,\ldots,a_k\}$ is a generating set for $A$.  An admissible $k$-tuple defines a cyclic covering space over a punctured sphere as follows:

Let $\Sigma_{0,k}$ be a sphere with $k$ punctures, and fix an enumeration of the punctures. Let $x_i$ be the homology class of a loop surrounding the $i$th puncture that is oriented counterclockwise.  The homomorphism $\phi:H_1(\Sigma_{0,k};\ZZ) \to A$ defined by $\phi(x_i) = a_i$ is surjective and therefore determines a regular cover of $\Sigma_{0,k}$ with deck group $A$.  By filling in the punctures we obtain a regular branched cover of $\Sigma_0$.

\begin{thm}\label{maintheorem}
Let $A$ be a finite cyclic group, and $(a_1,\ldots,a_k)$ an admissible $k$-tuple.  Let $\Sigma\rightarrow \Sigma_0$ be the cyclic branched cover of the sphere with deck group $A$ and $k$ branch points defined by the admissible tuple.  Every homeomorphism of $\Sigma_0$ lifts if and only if one of the following is true:
\begin{itemize}
\item There is an isomorphism $\delta:A \to \ZZ/n\ZZ$ with $k \equiv 0 \mod n$ such that $\delta (a_i) = 1$ for all $i$.
\item $k = 2$ and there is an isomorphism $\delta:A \to \ZZ/n\ZZ$ for some $n \geq 3$ such that $\delta(a_1) = 1$ and $\delta(a_2) = -1$.
\end{itemize}
\end{thm}

The 2-fold cover induced by the hyperelliptic involution is defined by the admissible tuple $(1,\ldots,1)$ with deck group equal to $\ZZ/2\ZZ$.  Our theorem provides an alternative proof that for the cover induced by the hyperelliptic involution, every homeomorphism lifts.

\p{Superelliptic Curves}  Choose distinct points $z_1,\ldots,z_t \in \CC$.  Any cyclic branched cover of the sphere can be modeled by an irreducible plane curve $C$ of the form:
\begin{equation}\label{superelliptic_curve}
y^n = (x - z_1)^{a_1} \cdots (x - z_t)^{a_t}
\end{equation}
for some $n \geq 2$ and integers $1 \leq a_i \leq n-1$.  Indeed, let $\wt C$ be the normalization of the projective closure of the affine curve $C$.  Projection onto the $x$ axis gives a $n$-sheeted cyclic branched covering $\wt C \to \PP^1$ that is branched at each $z_i$ and possibly at infinity.  There is branching at infinity if and only if $\sum_{i=1}^t a_i \not\equiv 0 \mod n$.

A cyclic branched covering space defined by (\ref{superelliptic_curve}) has deck group $A \cong \ZZ/n\ZZ$.  Such a cover is defined by the admissible tuple $(a_1,\ldots,a_t)$ if there is no branching at infinity.  If there is a branch point at infinity, then the cover is defined by $(a_1,\ldots,a_t, -\sum_{i=1}^t a_i)$ \cite[Ch 2]{Rohde}.  The irreducibility of $C$ ensures that the $\{a_i\}$ form a generating set for $A$.  We have an immediate corollary of Theorem \ref{maintheorem}.

\begin{cor} \label{lifting}
Let $\widetilde{C}\rightarrow\PP^1$ be a cyclic branched cover defined by an irreducible superelliptic curve as in equation (\ref{superelliptic_curve}).  Then every homeomorphism of $\PP^1$ lifts if and only if one of the following is true.
\begin{itemize}
\item $a_1 = \cdots = a_t$ and $t \equiv 0 \text{ or } -1 \mod n$,
\item $n \geq 3$ and $t = 1$, or
\item $n \geq 3$, $t = 2$ and $a_1 \equiv -a_2 \mod n$.
\end{itemize}
\end{cor}

\p{A cover where not all homeomorphisms lift}
We may represent a genus 2 surface $\Sigma_2$ as a 10-gon with sides identified as in figure \ref{not_all_lift}.  A counterclockwise rotation by $2\pi/5$ of the 10-gon induces an action of $\ZZ/5\ZZ$ on $\Sigma_2$.  The resulting quotient space is homeomorphic to a sphere.  The quotient map is a branched covering map $\Sigma_2 \to \Sigma_0$ branched at 3 points.  The preimages of the branch points in the 10-gon are the center and two distinct orbits of vertices under rotation.  The cover corresponds to the admissible tuple $(1,1,3)$ with the deck group equal to $\ZZ/5\ZZ$. Therefore, by Theorem \ref{maintheorem}, not all homeomorphisms of $\Sigma_0$ lift.

\begin{center}
\begin{tikzpicture}[>=stealth, scale=1.5]
	\def \x {36};
	\draw[decoration={markings,mark=at position .55 with {\arrow[scale=2]{<}}}, postaction={decorate}] ({cos(0)},{sin(0)}) -- node[right]{$a_1$} ({cos(\x)},{sin(\x)});
	\draw[decoration={markings,mark=at position .55 with {\arrow[scale=2]{>}}}, postaction={decorate}] ({cos(\x)},{sin(\x)}) -- node[label={[label distance=-.25cm]54:$a_5$}]{} ({cos(2*\x)},{sin(2*\x)});
	\draw[decoration={markings,mark=at position .55 with {\arrow[scale=2]{<}}}, postaction={decorate}] ({cos(2*\x)},{sin(2*\x)}) -- node[above]{$a_2$} ({cos(3*\x)},{sin(3*\x)});
	\draw[decoration={markings,mark=at position .55 with {\arrow[scale=2]{>}}}, postaction={decorate}] ({cos(3*\x)},{sin(3*\x)}) -- node[label={[label distance=-.25cm]126:$a_1$}]{} ({cos(4*\x)},{sin(4*\x)});
	\draw[decoration={markings,mark=at position .55 with {\arrow[scale=2]{<}}}, postaction={decorate}] ({cos(4*\x)},{sin(4*\x)}) -- node[left]{$a_3$} ({cos(5*\x)},{sin(5*\x)});
	\draw[decoration={markings,mark=at position .55 with {\arrow[scale=2]{>}}}, postaction={decorate}] ({cos(5*\x)},{sin(5*\x)}) -- node[left]{$a_2$} ({cos(6*\x)},{sin(6*\x)});
	\draw[decoration={markings,mark=at position .55 with {\arrow[scale=2]{<}}}, postaction={decorate}] ({cos(6*\x)},{sin(6*\x)}) -- node[label={[label distance=-.25cm]234:$a_4$}]{} ({cos(7*\x)},{sin(7*\x)});
	\draw[decoration={markings,mark=at position .55 with {\arrow[scale=2]{>}}}, postaction={decorate}] ({cos(7*\x)},{sin(7*\x)}) -- node[below]{$a_3$} ({cos(8*\x)},{sin(8*\x)});
	\draw[decoration={markings,mark=at position .55 with {\arrow[scale=2]{<}}}, postaction={decorate}] ({cos(8*\x)},{sin(8*\x)}) -- node[label={[label distance=-.25cm]306:$a_5$}]{} ({cos(9*\x)},{sin(9*\x)});
	\draw[decoration={markings,mark=at position .55 with {\arrow[scale=2]{>}}}, postaction={decorate}] ({cos(9*\x)},{sin(9*\x)}) -- node[right]{$a_4$} ({cos(10*\x)},{sin(10*\x)});
\end{tikzpicture}
\vspace{-.3cm}

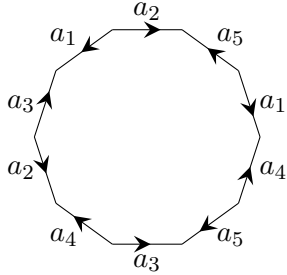
\captionof{figure}{A rotation by $2\pi/5$ generates an action of $\ZZ/5\ZZ$ on $\Sigma_2$.}
\label{not_all_lift}
\end{center}

\p{Application to the mapping class group}
Let $X$ be an orientable surface possibly with marked points.  The mapping class group of $X$, denoted $\Mod(X)$, is the group of orientation-preserving homeomorphisms of $X$ that preserve the set of marked points up to isotopies that preserve the set of marked points.

Let $\Sigma$ be a surface, and $\Sigma\rightarrow\Sigma_0$ be a cyclic branched cover with deck group $A$.  The symmetric mapping class group of $\Sigma$, denoted $\SMod(\Sigma)$, is the subgroup of $\Mod(\Sigma)$ comprised of isotopy classes of fiber-preserving homeomorphisms of $\Sigma$.  By work of Birman and Hilden \cite[Theorem 3]{BH}, the quotient $\SMod(\Sigma)/A$ is isomorphic to a finite-index subgroup of $\Mod(\Sigma_0)$.  When the conditions of Theorem \ref{maintheorem} are satisified, $\SMod(\Sigma)/A$ is isomorphic to $\Mod(\Sigma_0)$.   This gives us the following corrected statement of Theorem 5 in \cite{BH}.

\begin{thm} \label{theorem5}
Let $A$ be a finite cyclic group, and $(a_1,\ldots,a_k)$ an admissible $k$-tuple.  Let $\Sigma\rightarrow \Sigma_0$ be the cyclic branched cover of the sphere with deck group $A$ and $k$ branch points defined by the admissible tuple.  The quotient $\SMod(\Sigma)/A$ is isomorphic to $\Mod(\Sigma_0)$ if there is an isomorphism $\delta:A \to \ZZ/n\ZZ$ with $k \equiv 0 \mod n$ such that $\delta (a_i) = 1$ for all $i$, and $(n,k)$ is not equal to $(2,2)$, $(2,4)$, or $(3,3)$.
\end{thm}

Theorem 5 from \cite{BH} requires that $\Sigma$ is hyperbolic.  In the statement of Theorem \ref{theorem5}, we have excluded the cases from Theorem \ref{maintheorem} where $\Sigma$ is not hyperbolic.

\p{Acknowledgements} The authors would like to thank Joan Birman and Mike Hilden their encouragement and support. We would also like to thank Dan Margalit, David McKinnon, and Doug Park for helpful conversations regarding this paper.  We would also like to thank the referee for helpful comments.

\section{Proof of Main Theorem}
Let $\Sigma_{0,k}$ be a 2-sphere with $k$ punctures.  Any finite sheeted regular cover with base space $\Sigma_{0,k}$ and abelian deck group $A$ is determined by a surjective homomorphism $\phi:H_1(\Sigma_{0,k};\ZZ) \to A$.   The homomorphism $\phi$ is unique up to an automorphism of $A$.

\p{Reduction to the unbranched case}  A branched cover $\Sigma \to \Sigma_0$ with $k$ branch points induces a cover of $\Sigma_{0,k}$ by removing the branch points and their preimages in $\Sigma$.  Conversely, a cover of $\Sigma_{0,k}$ can be completed to a branched cover where each puncture is filled with a branch point. We will restrict homeomorphisms of $\Sigma_0$ to $\Sigma_{0,k}$ and extend homeomorphisms of $\Sigma_{0,k}$ to homeomorphisms of $\Sigma_0$ when it is convenient.

Fix an enumeration of the punctures and let $x_i$ be the homology class of a loop surrounding the $i$th puncture that is oriented counterclockwise.  Each $x_i \in H_1(\Sigma_{0,k};\ZZ)$ is supported on a neighborhood of the $i$th puncture.  Let $f$ be a homeomorphism of $\Sigma_{0,k}$.  The automorphism $f_*$ of $H_1(\Sigma_{0,k},\ZZ)$ is determined by the permutation $f$ induces on the punctures of $\Sigma_{0,k}$.  Indeed, let $\sigma \in S_k$ be the permutation induced by $f$.  If $f$ is orientation preserving, then $f_*(x_i) = x_{\sigma(i)}$.  If $f$ is orientation reversing, then $f_*(x_i) = -x_{\sigma(i)}$.

Recall that an admissible $k$-tuple defines a surjective homomorphism $\phi:H_1(\Sigma_{0,k};\ZZ) \to A$ by $\phi(x_i) = a_i$ and therefore defines a branched cover of $\Sigma_0$.  Another admisible $k$-tuple $(a_1',\ldots,a_k')$ determines an equivalent covering space if and only if there is an automorphism $\psi \in \Aut(A)$ such that $\psi(a_i) = a_i'$ for all $i$.

\begin{lem} \label{lifting_criterion}
Let $A$ be a finite abelian group, and $(a_1,\ldots, a_k)$ an admissible $k$-tuple.  Let $\Sigma \to \Sigma_{0,k}$ be the covering space defined by this tuple.  Let $f$ be a homeomorphism of $\Sigma_{0,k}$, and let $\sigma \in S_k$ be the permutation of the punctures induced by $f$.  The homeomorphism $f$ lifts if and only if there is an automorphism $\psi \in \Aut(A)$ such that $ \psi(a_i) = a_{\sigma(i)} $ for all $i$.
\end{lem}
\begin{proof}
Let $\phi:H_1(\Sigma_{0,k};\ZZ) \to A$ be the homomorphism defining the cover, which is defined by $\phi(x_i) = a_i$.  Let $f_*$ be the automorphism of $H_1(\Sigma_{0,k};\ZZ)$ induced by $f$.
We use the following facts:
\begin{enumerate}
\item The equality $f_*(\ker(\phi)) = \ker(\phi)$ holds if and only if $\ker(\phi f_*) = \ker(\phi)$.
\item Let $f,g:G \to A$ be surjective homomorphisms. Then  $\ker(f) = \ker(g)$ if and only if $f = \xi g$ for some $\xi \in \Aut(A)$.
\end{enumerate}
We omit the proofs of these facts.

The homeomorphism $f$ lifts if and only if $f_*(\ker(\phi)) = \ker(\phi)$.  Therefore by fact (1), $f$ lifts if and only if $\ker(\phi f_*)  =\ker(\phi)$. By fact (2), $f$ lifts if and only if there exists an automorphism $\psi \in \Aut(A)$ such that $\phi f_* = \psi \phi$.  If $f$ is orientation preserving, then $a_{\sigma(i)} = \phi f_* (x_i) = \psi \phi(x_i) = \psi (a_i)$ for all $i$ and the result follows.

The map $a \mapsto -a$ is an automorphism of an abelian group.  If $f$ is orientation reserving then we compose $a \mapsto -a$ with the automorphism $\psi$ in the orientation-preserving case.
\end{proof}

\begin{lem} \label{perm_aut}
Let $A$ be a finite cyclic group, and let $(a_1,\ldots,a_k)$ be an admissible tuple.  For every permutation $\sigma \in S_k$, there exists $\psi \in \Aut(A)$ such that $\psi(a_i) = a_{\sigma(i)}$ for all $i$ if and only if one of the following is true:
\begin{itemize}
\item There is an isomorphism $\delta:A \to \ZZ/n\ZZ$ with $k \equiv 0 \mod n$ such that $\delta (a_i) = 1$ for all $i$.
\item $k = 2$, and there is an isomorphism $\delta:A \to \ZZ/n\ZZ$ for some $n \geq 3$ such that $\delta(a_1) = 1$ and $\delta(a_2) = -1$.
\end{itemize}
\end{lem}
\begin{proof}
For the forward direction, first suppose all the $a_i$ are equal.  Then each $a_i$ is a generator of $A$, and there is an isomorphism $\delta: A \to \ZZ/n\ZZ$ such that $\delta(a_i) = 1$ for all $i$.  The condition that $\sum_{i=1}^k a_i = 0$ implies that $k \equiv 0 \mod n$.

Suppose now that the $a_i$ are not all equal.  Then they must all be distinct.  Indeed, assume to the contrary that there exist three distinct elements $a_p,a_q,a_r$ of the admissible tuple such that $a_p = a_q \neq a_r$.  Let $\sigma \in S_k$ be the transposition that switches $q$ and $r$.  By assumption there exists $\psi \in \Aut(A)$ such that $\psi(a_i) = a_{\sigma(i)}$ for all $i$.  Therefore $a_p = \psi(a_p) = \psi(a_q) = a_r$, which is a contradiction.

We may therefore assume the $a_i$ are distinct.  Then there is a subgroup of $\Aut(A)$ isomorphic to the symmetric group $S_k$.  Since the automorphism group of a cyclic group is abelian, it must be that $k=2$.  Since $k=2$, we have that $a_1 = -a_2$ with $a_1$ a generator of $A$.  Since $a_1$ and $a_2$ are distinct, $\abs A \geq 3$.   Therefore the map $\delta:A \to \ZZ/n\ZZ$ with $\delta(a_1)=1$ and $\delta(a_2)=-1$ is an isomorphism when $n=|A|$.

For the converse, we must write down an appropriate automorphism for each permutation $\sigma \in S_k$.  In the case that $\delta(a_i) = 1$ for all $i$, the identity automorphism suffices for all permutations.  In the case where $k = 2$, $\delta(a_1) = 1$, and $\delta(a_2) = -1$, the automorphism $a \mapsto -a$ of $A$ suffices for the non-trivial permutation.
\end{proof}

We now prove the main result.
\begin{proof}[Proof of Theorem \ref{maintheorem}]
Any permutation of the branch points can be induced by a homeomorphism of the sphere.  Therefore, the result follows by combining Lemmas \ref{lifting_criterion} and \ref{perm_aut}.
\end{proof}

\bibliography{LiftingHomeomorphisms}
\end{document}